\documentclass[11pt]{article}

\usepackage[letterpaper,margin=1.1in]{geometry}
\usepackage[T1]{fontenc}
\usepackage{amsmath,amssymb,amsthm}
\usepackage{newpxtext}
\usepackage{newpxmath}
\usepackage{mathtools}
\usepackage{microtype}
\usepackage[hidelinks]{hyperref}
\usepackage[nameinlink,noabbrev]{cleveref}
\usepackage{enumitem}

\numberwithin{equation}{section}
\allowdisplaybreaks

\newtheorem{theorem}{Theorem}[section]
\newtheorem{proposition}[theorem]{Proposition}

\newtheorem{corollary}[theorem]{Corollary}
\theoremstyle{definition}
\newtheorem{definition}[theorem]{Definition}

\theoremstyle{remark}
\newtheorem{remark}[theorem]{Remark}

\crefname{theorem}{theorem}{theorems}
\crefname{proposition}{proposition}{propositions}
\crefname{lemma}{lemma}{lemmas}
\crefname{corollary}{corollary}{corollaries}
\crefname{definition}{definition}{definitions}
\crefname{example}{example}{examples}
\crefname{remark}{remark}{remarks}

\newcommand{\SP}{\mathrm{SP}}

\newcommand{\vac}{\mathbf{1}}
\newcommand{\cQ}{\mathcal{Q}}
\newcommand{\Yop}{\mathcal{Y}}
\newcommand{\Ht}{\mathsf{H}}
\newcommand{\Rt}{\mathcal{R}_t}
\newcommand{\Pf}{\operatorname{Pf}}

\newcommand{\QQ}{\mathbb{Q}}
\newcommand{\ZZ}{\mathbb{Z}}
\newcommand{\EE}{\mathbb{E}}
\newcommand{\Var}{\operatorname{Var}}

\title{A Shifted \texorpdfstring{$t$}{t}-Schur Weight from the Modified Odd Operator}
\author{S.-J. Lee}
\date{}

\begin{document}

\maketitle

\begin{abstract}
We study the one-time weight on strict partitions obtained from the
modified odd Greaves--Jing--Zhu operator.  The shifted $t$-Schur
functions generated by this operator are obtained from the classical
Schur $Q$-functions by the plethystic substitution $X\mapsto X-tX$.
Thus the corresponding weight
\[
        \lambda\longmapsto \cQ_\lambda(X;t)P_\lambda(Y)
\]
is a shifted Schur weight with a virtual first alphabet.  We give its
normalization, its Pfaffian correlation kernel, its Fredholm Pfaffian for
the largest part, and its size cumulants.  For $t=-q$ with $q\geq 0$ the
virtual alphabet becomes the positive alphabet $X+qX$, giving a genuine
probability measure.  This positive specialization is the one-time
marginal of the two-color lift considered in a companion note.
\end{abstract}

\tableofcontents

\section{Introduction}

Greaves, Jing, and Zhu introduced a vertex-operator construction for the
$t$-Schur functions and used it to study the $t$-Schur measure
\cite{GJZ}.  The modified odd version of their operator produces
symmetric functions indexed by strict partitions; these shifted
$t$-Schur functions were introduced in \cite{LeeGessel}.  They
specialize to the Schur $Q$-functions when $t=0$ and satisfy odd
analogues of the basic formulas in Schur $Q$-theory.  The diagonal
change-of-variables viewpoint for the modified operator appears in
\cite{LeeMixed}, and transition matrices between shifted $t$-Schur bases
are studied in \cite{LeeTransitions}.

The purpose of this article is to isolate the one-time probabilistic
object attached to the modified odd operator.  The construction
identifies this object with a shifted Schur measure whose first
specialization is the virtual alphabet $X-tX$.  This identification
transfers the Pfaffian correlation structure of the shifted Schur measure
to the shifted $t$-Schur weight.

The key point is the plethystic identity
\begin{equation}\label{eq:intro-plethysm}
        \cQ_\lambda(X;t)=Q_\lambda[X-tX],
        \qquad p_n[X-tX]=(1-t^n)p_n(X)\quad(n\text{ odd}).
\end{equation}
Consequently, the one-time shifted $t$-Schur weight is
\begin{equation}\label{eq:intro-weight}
        W_t^{X,Y}(\lambda)=\cQ_\lambda(X;t)P_\lambda(Y)
        =Q_\lambda[X-tX]P_\lambda(Y).
\end{equation}
For general $t$ this is a formal weight, because $X-tX$ is a virtual
alphabet.  When $t=-q$ and $q\geq0$, however,
\[
        X-tX=X+qX,
\]
and the weight becomes positive for nonnegative specializations $X$ and
$Y$ satisfying the usual convergence conditions.

The main formulas are as follows.  The normalizing constant is
\begin{equation}\label{eq:intro-Zt}
        Z_t(X,Y)
        =\exp\left(
        2\sum_{\substack{n\geq1\\ n\text{ odd}}}
        \frac{(1-t^n)p_n(X)p_n(Y)}{n}
        \right).
\end{equation}
For finite alphabets this is
\begin{equation}\label{eq:intro-Zt-product}
        Z_t(X,Y)
        =\prod_{i,j}
        \frac{(1+x_i y_j)(1-tx_i y_j)}{(1-x_i y_j)(1+tx_i y_j)}.
\end{equation}
The associated point process on $\ZZ_{>0}$ is Pfaffian, with Matsumoto's
shifted-Schur kernel specialized to the symbol
\begin{equation}\label{eq:intro-Jt}
        J_t(z;X,Y)=F_{X-tX}(z)F_Y(-z^{-1}).
\end{equation}
Here
\[
        F_A(z)=\exp\left(
        2\sum_{\substack{n\geq1\\n\text{ odd}}}
        \frac{p_n(A)}{n}z^n\right).
\]
Finally, the cumulants of the size $|\lambda|$ are
\begin{equation}\label{eq:intro-cumulants}
        \kappa_m(|\lambda|)
        =2\sum_{\substack{n\geq1\\ n\text{ odd}}}
        n^{m-1}(1-t^n)p_n(X)p_n(Y).
\end{equation}
At $t=-q$, this formula splits into the two contributions corresponding
to the alphabets $X$ and $qX$.

\section{Schur \texorpdfstring{$Q$}{Q}-functions and plethystic notation}

We use standard notation for Schur $Q$-functions; see
\cite[Chapter III, Section 8]{Macdonald}.  Let
$\Gamma=\QQ[p_1,p_3,p_5,\ldots]$ be the ring of symmetric functions
generated by the odd power sums.  We write $\SP$ for the set of strict
partitions, including the empty partition.

Unless analytic convergence is imposed, all generating series below are
understood in the degree completion of the odd power-sum ring, or in the
corresponding completed tensor product in the $X$- and $Y$-variables.

For an alphabet or specialization $A$, define
\begin{equation}\label{eq:FA}
        F_A(z)
        =\prod_{a\in A}\frac{1+az}{1-az}
        =\exp\left(
        2\sum_{\substack{n\geq1\\n\text{ odd}}}
        \frac{p_n(A)}{n}z^n
        \right)
        =\sum_{r\geq0}q_r(A)z^r.
\end{equation}
The Schur $Q$-function $Q_\lambda(A)$ may be defined by the generating
series
\begin{equation}\label{eq:Q-generating}
        \prod_{i=1}^m F_A(z_i)
        \prod_{1\leq i<j\leq m}\frac{z_i-z_j}{z_i+z_j}
        =\sum_{\lambda\in\SP,\, \ell(\lambda)\leq m}
        Q_\lambda(A)z_1^{\lambda_1}\cdots z_m^{\lambda_m}
\end{equation}
whenever $m$ is large enough.  Equivalently, for $r>s\geq0$,
\begin{equation}\label{eq:two-row-classical}
        Q_{(r,s)}(A)
        =q_r(A)q_s(A)
        +2\sum_{i=1}^s(-1)^i q_{r+i}(A)q_{s-i}(A),
\end{equation}
with $Q_{(r,s)}=-Q_{(s,r)}$ when $r<s$ and $Q_{(r,r)}=0$, and the
Pfaffian Giambelli formula gives
\begin{equation}\label{eq:pfaffian-giambelli}
        Q_\lambda(A)=\Pf\left(Q_{(\lambda_i,\lambda_j)}(A)\right).
\end{equation}
If $\ell(\lambda)$ is odd, we append a zero part before taking the
Pfaffian.  We use the standard convention
\begin{equation}\label{eq:P-Q-relation}
        P_\lambda(A)=2^{-\ell(\lambda)}Q_\lambda(A).
\end{equation}
The Cauchy identity is
\begin{equation}\label{eq:Cauchy-QP}
        \sum_{\lambda\in\SP}Q_\lambda(A)P_\lambda(B)
        =\Ht(A;B),
\end{equation}
where
\begin{equation}\label{eq:HAB}
        \Ht(A;B)
        =\exp\left(
        2\sum_{\substack{n\geq1\\ n\text{ odd}}}
        \frac{p_n(A)p_n(B)}{n}\right).
\end{equation}
For finite alphabets,
\begin{equation}\label{eq:HAB-product}
        \Ht(A;B)=\prod_{a\in A,\, b\in B}\frac{1+ab}{1-ab}.
\end{equation}

We use plethystic notation throughout.  Thus $A+B$ denotes the
specialization determined by
\[
        p_n[A+B]=p_n[A]+p_n[B],
\]
and $cA$ denotes the specialization with
\[
        p_n[cA]=c^np_n[A].
\]
In particular,
\begin{equation}\label{eq:Xminus-tX}
        p_n[X-tX]=(1-t^n)p_n(X).
\end{equation}

\section{The modified odd operator}

Let
\[
        \Gamma_t=\QQ(t)[p_1,p_3,p_5,\ldots].
\]
Define the diagonal operator
\begin{equation}\label{eq:Rt-definition}
        \Rt(p_n)=(1-t^n)p_n,\qquad n\geq1\text{ odd},
\end{equation}
and extend it multiplicatively and linearly to $\Gamma_t$.  Its inverse
over $\QQ(t)$ is given by
\[
        \Rt^{-1}(p_n)=\frac{p_n}{1-t^n},\qquad n\geq1\text{ odd}.
\]

Let $\Phi(z)$ denote the classical Schur $Q$ vertex operator on the odd
power-sum ring,
\begin{equation}\label{eq:classical-Q-vertex}
        \Phi(z)
        =\exp\left(
        2\sum_{\substack{n\geq1\\ n\text{ odd}}}
        \frac{p_n}{n}z^n\right)
        \exp\left(
        -\sum_{\substack{n\geq1\\ n\text{ odd}}}
        \frac{\partial}{\partial p_n}z^{-n}\right).
\end{equation}
The modified Greaves--Jing--Zhu operator in the odd sector is
\begin{equation}\label{eq:modified-operator}
        \Yop_t(z)
        =\exp\left(
        2\sum_{\substack{n\geq1\\ n\text{ odd}}}
        \frac{(1-t^n)p_n}{n}z^n\right)
        \exp\left(
        -\sum_{\substack{n\geq1\\ n\text{ odd}}}
        \frac{1}{1-t^n}\frac{\partial}{\partial p_n}z^{-n}\right).
\end{equation}

\begin{proposition}[Conjugation form]\label{prop:conjugation}
The modified operator satisfies
\begin{equation}\label{eq:Yt-conjugation}
        \Yop_t(z)=\Rt\Phi(z)\Rt^{-1}.
\end{equation}
Consequently, the shifted $t$-Schur functions generated by the modified
operator are
\begin{equation}\label{eq:Qt-plethystic}
        \cQ_\lambda(X;t)=\Rt Q_\lambda(X)=Q_\lambda[X-tX].
\end{equation}
\end{proposition}

\begin{proof}
From the definition of $\Rt$,
\[
        \Rt p_n\Rt^{-1}=(1-t^n)p_n.
\]
Also, since $\Rt$ rescales $p_n$ by $1-t^n$,
\[
        \Rt\frac{\partial}{\partial p_n}\Rt^{-1}
        =\frac{1}{1-t^n}\frac{\partial}{\partial p_n}.
\]
Substituting these two relations into \eqref{eq:classical-Q-vertex}
gives \eqref{eq:Yt-conjugation}.

Since $\Rt^{-1}\vac=\vac$, for every $m\geq1$ we have
\[
        \Yop_t(z_1)\cdots \Yop_t(z_m)\vac
        =
        \Rt\Phi(z_1)\cdots\Phi(z_m)\vac .
\]
Taking the coefficient corresponding to a strict partition $\lambda$ in
the same convention as the classical Schur $Q$ vertex construction gives
\[
        \cQ_\lambda(X;t)=\Rt Q_\lambda(X).
\]
The operator $\Rt$ rescales each odd power sum by
$p_n\mapsto (1-t^n)p_n$, which is exactly the plethystic substitution
$X\mapsto X-tX$.  Hence
\[
        \Rt Q_\lambda(X)=Q_\lambda[X-tX].
\]
\end{proof}

\begin{remark}\label{rem:same-structure-constants}
The operator $\Rt$ is a graded Hopf algebra automorphism over
$\QQ(t)$.  Hence the multiplicative structure constants of the basis
$\{\cQ_\lambda(X;t)\}$ are the same as those of the classical Schur
$Q$-basis after the diagonal change of variables.  The $t$-dependence is
encoded by the plethystic specialization of the alphabets and by the
corresponding shifted Schur weight.
\end{remark}

\section{The shifted \texorpdfstring{$t$}{t}-Schur weight}

\begin{definition}\label{def:shifted-t-weight}
For specializations $X$ and $Y$, define the shifted $t$-Schur weight on
strict partitions by
\begin{equation}\label{eq:shifted-t-weight}
        W_t^{X,Y}(\lambda)=\cQ_\lambda(X;t)P_\lambda(Y)
        =Q_\lambda[X-tX]P_\lambda(Y).
\end{equation}
Its normalizing series is
\begin{equation}\label{eq:Zt-definition}
        Z_t(X,Y)=\sum_{\lambda\in\SP}W_t^{X,Y}(\lambda).
\end{equation}
Whenever $Z_t(X,Y)$ is invertible as a formal series, we write
\begin{equation}\label{eq:formal-measure}
        \mathbb{P}_t^{X,Y}(\lambda)=\frac{W_t^{X,Y}(\lambda)}{Z_t(X,Y)}.
\end{equation}
For general $t$ this is a formal normalized weight.  In the positive
specialization described in \cref{cor:positive-specialization}, it is a
probability measure.
\end{definition}

\begin{theorem}[Normalization]\label{thm:normalization}
The normalizing series is
\begin{equation}\label{eq:Zt-H}
        Z_t(X,Y)=\Ht(X-tX;Y).
\end{equation}
Equivalently,
\begin{equation}\label{eq:Zt-exponential}
        Z_t(X,Y)
        =\exp\left(
        2\sum_{\substack{n\geq1\\n\text{ odd}}}
        \frac{(1-t^n)p_n(X)p_n(Y)}{n}\right).
\end{equation}
If $X=(x_1,x_2,\ldots)$ and $Y=(y_1,y_2,\ldots)$ are finite or
absolutely convergent alphabets, then
\begin{equation}\label{eq:Zt-product-main}
        Z_t(X,Y)
        =\prod_{i,j}
        \frac{(1+x_i y_j)(1-tx_i y_j)}{(1-x_i y_j)(1+tx_i y_j)}.
\end{equation}
\end{theorem}

\begin{proof}
By \cref{prop:conjugation} and the Schur $Q/P$ Cauchy identity,
\[
        Z_t(X,Y)
        =\sum_{\lambda\in\SP}Q_\lambda[X-tX]P_\lambda(Y)
        =\Ht(X-tX;Y).
\]
The exponential expression follows from the definition of $\Ht$ and
$p_n[X-tX]=(1-t^n)p_n(X)$.  For finite alphabets,
\[
        \Ht(X-tX;Y)=\Ht(X;Y)\Ht(-tX;Y).
\]
Using \eqref{eq:HAB-product} and
\[
        \Ht(-tX;Y)=\prod_{i,j}\frac{1-tx_i y_j}{1+tx_i y_j}
\]
gives \eqref{eq:Zt-product-main}.
\end{proof}

\begin{corollary}[Positive specialization]\label{cor:positive-specialization}
Let $t=-q$ with $q\geq0$.  If $X$ and $Y$ are nonnegative finite
alphabets satisfying
\begin{equation}\label{eq:convergence-positive}
        \max\{1,q\}x_i y_j<1
        \qquad\text{for all }i,j,
\end{equation}
then
\begin{equation}\label{eq:positive-measure}
        \mathbb{P}_{-q}^{X,Y}(\lambda)
        =\frac{Q_\lambda[X+qX]P_\lambda(Y)}{\Ht(X+qX;Y)}
\end{equation}
is a probability measure on $\SP$.
\end{corollary}

\begin{proof}
For odd $n$ we have $1-(-q)^n=1+q^n$, so $X-tX=X+qX$ at $t=-q$.
This is the sum of two nonnegative specializations, $X$ and $qX$.
Schur $Q$- and $P$-functions are nonnegative on nonnegative
specializations, and \eqref{eq:convergence-positive} guarantees that the
Cauchy product is finite.
\end{proof}

\section{Pfaffian correlations}

We identify strict partitions with finite subsets of $\ZZ_{>0}$:
\[
        \lambda=(\lambda_1>\lambda_2>\cdots>\lambda_\ell>0)
        \quad\longleftrightarrow\quad
        \{\lambda_1,\ldots,\lambda_\ell\}.
\]
For a finite subset $S\subset\ZZ_{>0}$, define the correlation function
\begin{equation}\label{eq:correlation-function}
        \rho_t(S)=\mathbb{P}_t^{X,Y}\{\lambda\in\SP:S\subseteq\lambda\}.
\end{equation}
This is a formal correlation function unless the weight is positive.

\begin{definition}\label{def:Matsumoto-kernel}
Set
\begin{equation}\label{eq:Jt-definition}
        J_t(z;X,Y)=F_{X-tX}(z)F_Y(-z^{-1}).
\end{equation}
For $u,v\in\ZZ\setminus\{0\}$, let $K_t(u,v)$ be $\varepsilon(u,v)$
times the coefficient of $z^u w^v$ in the formal Laurent expansion of
\begin{equation}\label{eq:kernel-generating}
        \frac12 J_t(z;X,Y)J_t(w;X,Y)\frac{z-w}{z+w}.
\end{equation}
The expansion convention is the following: $F_{X-tX}(z)$ is expanded in
nonnegative powers of $z$, $F_Y(-z^{-1})$ is expanded in nonpositive
powers of $z$, and
\begin{equation}\label{eq:kernel-expansion-convention}
        \iota_{z,w}\frac{z-w}{z+w}
        =
        \frac{1-w/z}{1+w/z}
        =
        1+2\sum_{m\geq1}(-1)^m\left(\frac{w}{z}\right)^m.
\end{equation}
Here $\iota_{z,w}$ denotes expansion in the region $|w|<|z|$.

The sign factor is
\begin{equation}\label{eq:epsilon}
        \varepsilon(u,v)=
        \begin{cases}
        1, & u>0,\ v>0,\\
        (-1)^v, & u>0,\ v<0,\\
        (-1)^{u+v}, & u<0,\ v<0.
        \end{cases}
\end{equation}
The remaining values are determined by skew-symmetry,
$K_t(v,u)=-K_t(u,v)$.
\end{definition}

\begin{theorem}[Pfaffian correlation formula]\label{thm:pfaffian-correlation}
Let $S=\{k_1,\ldots,k_N\}\subset\ZZ_{>0}$.  Define the $2N\times2N$
skew-symmetric matrix $M_t(S)$ by specifying its entries above the
diagonal:
\begin{equation}\label{eq:M-matrix}
        M_t(S)_{ij}=
        \begin{cases}
        K_t(k_i,k_j), & 1\leq i<j\leq N,\\
        K_t(k_i,-k_{2N-j+1}), & 1\leq i\leq N<j\leq2N,\\
        K_t(-k_{2N-i+1},-k_{2N-j+1}), & N<i<j\leq2N.
        \end{cases}
\end{equation}
Then
\begin{equation}\label{eq:pfaffian-correlation}
        \rho_t(S)=\Pf M_t(S).
\end{equation}
\end{theorem}

\begin{proof}
Matsumoto's Pfaffian formula for the shifted Schur measure applies to
any pair of specializations $(A,Y)$ for which the Cauchy series is well
defined, and also as a formal identity \cite{Matsumoto}.  We apply that
formula with the first specialization equal to the virtual specialization
\[
        A=X-tX.
\]
The shifted Schur measure with parameters $(A,Y)$ has weight
$Q_\lambda(A)P_\lambda(Y)$.  By \cref{prop:conjugation}, this is exactly
$\cQ_\lambda(X;t)P_\lambda(Y)$.  The symbol in Matsumoto's theorem is
$F_A(z)F_Y(-z^{-1})$, which becomes \eqref{eq:Jt-definition}.  This gives
\eqref{eq:pfaffian-correlation}.
\end{proof}

\begin{remark}\label{rem:kernel-signs}
The sign convention in \cref{def:Matsumoto-kernel} is the one used for
the shifted Schur measure.  It is often convenient in process formulas
to use Vuleti\'c's time-dependent convention \cite{Vuletic}.  The
one-time and process symbols are equivalent after the standard change of
variables and signs, but the conventions should be kept separate in
explicit formulas.
\end{remark}

\section{Largest part and size statistics}

\subsection{Fredholm Pfaffian for the largest part}

Let $h\geq0$.  The event $\{\lambda_1\leq h\}$ is the event that the
point configuration has no point in $\{h+1,h+2,\ldots\}$.  The Pfaffian
correlation formula therefore gives the usual Fredholm Pfaffian
expansion
\begin{equation}\label{eq:fredholm-pfaffian-largest}
        \mathbb{P}_t^{X,Y}(\lambda_1\leq h)
        =\Pf\bigl(\mathbf J-\mathbf K_t\bigr)_{\ell^2(\{h+1,h+2,\ldots\})},
\end{equation}
where the $2\times2$ block kernel is
\begin{equation}\label{eq:block-kernel}
        \mathbf K_t(r,s)
        =
        \begin{pmatrix}
        K_t(r,s) & K_t(r,-s)\\
        K_t(-r,s) & K_t(-r,-s)
        \end{pmatrix},
        \qquad r,s>h,
\end{equation}
and $\mathbf J(r,s)=\delta_{rs}
\begin{psmallmatrix}0&1\\-1&0\end{psmallmatrix}$ is the standard skew
identity kernel.  In expanded form,
\begin{equation}\label{eq:fredholm-expanded}
        \mathbb{P}_t^{X,Y}(\lambda_1\leq h)
        =
        1+\sum_{m\geq1}(-1)^m
        \sum_{h<r_1<\cdots<r_m}
        \rho_t(\{r_1,\ldots,r_m\}).
\end{equation}
For positive specializations the series converges under the usual
trace-class assumptions; otherwise it is read as a formal identity.  This
is the shifted Schur-measure Fredholm Pfaffian specialized to
$X-tX$; compare \cite{TW,Matsumoto,Vuletic}.

\subsection{Size generating function}

\begin{theorem}[Size cumulants]\label{thm:size-cumulants}
Let $|\lambda|$ denote the size of a strict partition distributed
according to the formal normalized weight \eqref{eq:formal-measure}. Then
\begin{equation}\label{eq:size-pgf}
        \EE_t[u^{|\lambda|}]
        =\frac{Z_t(uX,Y)}{Z_t(X,Y)}.
\end{equation}
Equivalently,
\begin{equation}\label{eq:size-log-mgf}
        \log\EE_t[e^{s|\lambda|}]
        =2\sum_{\substack{n\geq1\\n\text{ odd}}}
        \frac{(e^{ns}-1)(1-t^n)p_n(X)p_n(Y)}{n}.
\end{equation}
Hence the $m$th cumulant of $|\lambda|$ is
\begin{equation}\label{eq:size-cumulants-main}
        \kappa_m(|\lambda|)
        =2\sum_{\substack{n\geq1\\n\text{ odd}}}
        n^{m-1}(1-t^n)p_n(X)p_n(Y).
\end{equation}
In particular,
\begin{equation}\label{eq:size-mean-variance}
        \EE_t|\lambda|
        =
        2\sum_{\substack{n\geq1\\n\text{ odd}}}
        (1-t^n)p_n(X)p_n(Y),
\end{equation}
\begin{equation}\label{eq:size-variance}
        \Var_t(|\lambda|)
        =
        2\sum_{\substack{n\geq1\\n\text{ odd}}}
        n(1-t^n)p_n(X)p_n(Y).
\end{equation}
\end{theorem}

\begin{proof}
The Schur $Q$-function $Q_\lambda$ is homogeneous of degree $|\lambda|$.
Thus multiplying the first alphabet $X$ by $u$ multiplies the weight of
$\lambda$ by $u^{|\lambda|}$:
\[
        u^{|\lambda|}Q_\lambda[X-tX]
        =Q_\lambda[uX-tuX].
\]
Therefore
\[
        \EE_t[u^{|\lambda|}]
        =\frac{\sum_\lambda Q_\lambda[uX-tuX]P_\lambda(Y)}{Z_t(X,Y)}
        =\frac{Z_t(uX,Y)}{Z_t(X,Y)}.
\]
Substituting \eqref{eq:Zt-exponential} and then setting $u=e^s$ gives
\eqref{eq:size-log-mgf}.  The cumulants are obtained by applying
$\partial_s^m$ at $s=0$.
\end{proof}

\begin{corollary}\label{cor:size-negative-q}
At $t=-q$ with $q\geq0$,
\begin{equation}\label{eq:size-cumulants-negative-q}
        \kappa_m(|\lambda|)
        =
        2\sum_{\substack{n\geq1\\n\text{ odd}}}
        n^{m-1}(1+q^n)p_n(X)p_n(Y).
\end{equation}
\end{corollary}

\section{Rectangular specialization}

Let
\begin{equation}\label{eq:rectangular-specialization}
        X=(\underbrace{x,\ldots,x}_{M}),
        \qquad
        Y=(\underbrace{y,\ldots,y}_{N}),
\end{equation}
and set $L=MN$.  Then
\begin{equation}\label{eq:rectangular-Z}
        Z_t
        =
        \left(
        \frac{(1+xy)(1-txy)}{(1-xy)(1+txy)}
        \right)^L.
\end{equation}
The Pfaffian symbol becomes
\begin{equation}\label{eq:rectangular-Jt}
        J_t(z)
        =
        \left(\frac{1+xz}{1-xz}\right)^M
        \left(\frac{1-txz}{1+txz}\right)^M
        \left(\frac{1-yz^{-1}}{1+yz^{-1}}\right)^N.
\end{equation}
Equivalently, after clearing negative powers in the last factor,
\begin{equation}\label{eq:rectangular-Jt-rational}
        J_t(z)
        =
        \left(\frac{1+xz}{1-xz}\right)^M
        \left(\frac{1-txz}{1+txz}\right)^M
        \left(\frac{z-y}{z+y}\right)^N.
\end{equation}
The scalar kernel is the coefficient of \eqref{eq:kernel-generating}; in
regions where contour integral notation is justified, it can be written
as
\begin{equation}\label{eq:rectangular-integral}
        K_t(u,v)
        =
        \varepsilon(u,v)\frac{1}{(2\pi i)^2}
        \oint\oint
        \frac12 J_t(z)J_t(w)\frac{z-w}{z+w}
        \frac{dz\,dw}{z^{u+1}w^{v+1}}.
\end{equation}
The contours are chosen to realize the same Laurent expansion as in
\cref{def:Matsumoto-kernel}.  In the positive specialization $t=-q$, the
poles coming from the first alphabet are at
\[
        z=\frac1x,
        \qquad
        z=\frac1{qx}\quad(q>0),
\]
while the pole coming from $F_Y(-z^{-1})$ is at $z=-y$.  The convergence
condition $\max\{1,q\}xy<1$ permits nested circular contours realizing
the expansion \eqref{eq:kernel-expansion-convention}; for example, for
$q>0$ one may choose radii
\[
        y<R_w<R_z<\min\left\{\frac1x,\frac1{qx}\right\},
\]
and for $q=0$ one may choose
\[
        y<R_w<R_z<\frac1x.
\]
Thus the contours enclose the pole at $-y$, avoid the poles at $1/x$ and
$1/(qx)$, and satisfy $|w|<|z|$.

The size cumulants reduce to
\begin{equation}\label{eq:rectangular-cumulants}
        \kappa_m(|\lambda|)
        =
        2L\sum_{\substack{n\geq1\\ n\text{ odd}}}
        n^{m-1}(1-t^n)(xy)^n.
\end{equation}
In particular, for $t=-q$,
\begin{equation}\label{eq:rectangular-cumulants-negative-q}
        \kappa_m(|\lambda|)
        =
        2L\sum_{\substack{n\geq1\\ n\text{ odd}}}
        n^{m-1}(1+q^n)(xy)^n.
\end{equation}

\section{The bridge to the two-color lift}

At the positive specialization $t=-q$,
\begin{equation}\label{eq:bridge-plethysm}
        \cQ_\lambda(X;-q)=Q_\lambda[X+qX].
\end{equation}
The creation symbol factorizes as
\begin{equation}\label{eq:bridge-symbol-factorization}
        F_{X+qX}(z)=F_X(z)F_{qX}(z).
\end{equation}
Equivalently, if $\Gamma_-(A)$ denotes the Schur $Q$ creation half-vertex
operator with symbol $F_A(z)$, then
\begin{equation}\label{eq:bridge-factorization}
        \Gamma_-^{(-q)}(X)=\Gamma_-(X)\Gamma_-(qX).
\end{equation}

On the symmetric-function side, the same factorization is expressed by
the Schur $Q$ addition formula
\begin{equation}\label{eq:bridge-addition}
        Q_\lambda[X+qX]
        =
        \sum_{\mu\subseteq\lambda}
        Q_{\lambda/\mu}(X)Q_\mu(qX),
\end{equation}
where $Q_{\lambda/\mu}$ denotes the standard skew Schur $Q$-function.
Thus the positive specialization admits the two-color refinement
\begin{equation}\label{eq:bridge-two-color-weight}
        W^{X,Y,q}(\mu,\lambda)
        =
        Q_{\lambda/\mu}(X)Q_\mu(qX)P_\lambda(Y),
        \qquad
        \mu\subseteq\lambda,\quad \mu,\lambda\in\SP.
\end{equation}
Summing over the intermediate strict partition $\mu$ gives
\[
        \sum_{\mu\subseteq\lambda}W^{X,Y,q}(\mu,\lambda)
        =
        Q_\lambda[X+qX]P_\lambda(Y),
\]
which is the one-time shifted $t$-Schur weight at $t=-q$.  The
two-color lift retains the intermediate strict partition $\mu$ instead
of summing it out.

\end{document}